\documentclass[11pt]{amsproc}
\usepackage{mathrsfs}
\usepackage{stmaryrd}
\usepackage{cases}
\usepackage{amssymb}
\usepackage{amsmath}
\usepackage{amsfonts}
\usepackage{graphicx}
\usepackage{amsmath,amstext,amsbsy,amssymb}

\newtheorem{theorem}{Theorem}[section]
\newtheorem{lemma}[theorem]{Lemma}

\newtheorem{proposition}[theorem]{Proposition}
\newtheorem{corollary}[theorem]{Corollary}

\theoremstyle{definition}

\theoremstyle{remark}
\newtheorem{remark}[theorem]{Remark}

\numberwithin{equation}{section} \errorcontextlines=0

\newcommand{\Pf}{\mbox{Pf}}

\begin{document}

\title[Quantum Pfaffians and Hyper-Pfaffians]
{Quantum Pfaffians and Hyper-Pfaffians}
\author{Naihuan Jing}
\address{Department of Mathematics,
   North Carolina State University,
   Raleigh, NC 27695, USA}
\email{jing@math.ncsu.edu}
\author{Jian Zhang}
\address{School of Sciences,
South China University of Technology, Guangzhou, Guangdong 510640, China}
\email{zhang.jian@scut.edu.cn}
\thanks{*Corresponding author: Jian Zhang}

\begin{abstract}

The concept of the quantum Pfaffian is rigorously examined and refurbished
using the new method of quantum exterior algebras.
We derive a complete family of Pl\"ucker relations for the quantum linear transformations, and then use them to give
an optimal set of relations required for the
quantum Pfaffian. We then give the formula between the quantum determinant and the quantum Pfaffian
and prove that any quantum determinant can be expressed as a quantum Pfaffian. Finally
the quantum hyper-Pfaffian is introduced, and we prove a similar
 result of expressing quantum determinants in terms of quantum hyper-Pfaffians
 at modular cases.

\end{abstract}
\maketitle
\section{Introduction}

Quantum groups are certain deformation of
the coordinate rings of the underlying Lie groups. They first appeared
in the framework of
the quantum inverse scattering method of the St. Petersberg school \cite{KS, FRT}, where
a more general deformation of $U(sl(2))$--the Sklyanin algebra was formulated
as one of the first examples.
In the nontrivial example of the quantum coordinate ring $\mbox{GL}_q(n)$ of the
general linear group \cite{AST, BG, GL, KlS, LS, J, PW}, the quantum determinant (cf. \cite{NYM, HZ}) serves as a distinguished central element
in the quantum group and it is well-known that lots of properties of the determinant can
be generalized to the quantum case (cf. \cite{TT, MT}). However when the quantum matrix is anti-symmetric,
the quantum determinant is not exactly a square element. So the classical definition
of the Pfaffian as a square root of the determinant does not work in the quantum case.
In \cite{S} a notion of a quantum Pfaffian was introduced in the context of quantum invariant
theory, but
it was not clear how it was related to the quantum determinant from the
context.

In \cite{N, JR} the invariant theory of certain quantum symplectic group was used to define
the notion of quantum anti-symmetric matrices, where
the quantum Pfaffian can be realized on the quantum coordinate ring as well, and
Ray and the first author proved that the quantum Pfaffian is equal to the quantum determinant
after change of variables. This relation between the quantum determinant
and quantum Pfaffian was derived by making use of the representation theory
of the quantum algebra. One could then ask an even more fundamental question
if it is really necessary to use the representation theory to find such an identity for the quantum Pfaffian.
Moreover when we derived the determinant identity for the quantum Pfaffian in \cite{JR} we had used certain
quadratic relations from the quantum symplectic group, and we had conjectured that they are not
necessary for the determinant identity.

In this work we will study the quantum Pfaffian using quantum exterior algebras
and re-establish the
relationship between $q$-determinant and $q$-Pfaffian using fewer relations for the
generators.
In \cite{Ma} Manin constructed the quantum matrix algebra $\mbox{GL}_q(2)$ as a transformation group preserving
a pair of dual quantum planes. Later this method was generalized to higher dimension and in \cite{Ma2}
the notion of quantum deRham complexes was introduced based on quantum exterior algebras.
We will adopt this approach to study quantum
determinants and quantum Pfaffians, and prove that the quantum determinant can be viewed
as the volume element for the de Rham $q$-form. Then we derive the quantum Pl\"ucker relations and use them
as the starting point for our discussion of quantum Pfaffians and hyper-Pfaffians.
We prove that the quantum Pfaffian can be indeed defined for a more general classes of $q$-antisymmetric matrices
and we prove some non-trivial identities between the quantum determinant and the quantum Pfaffian. In particular,
we show that it is not necessary to impose more quadratic relations on the generators for the
quantum Pfaffian. Just as in the classical case \cite{H} (see also \cite{L}) we also prove that
the quantum Pfaffian is more fundamental than the quantum determinant in the sense that
any $q$-determinant can be
expressed as a $q$-Pfaffian.

In the last section we introduce the notion of quantum hyperdeterminant generalizing Luque-Thibon's work
in the classical case \cite{LT}. An interesting new phenomenon appears that the quantum hyper-Pfaffian satisfies the non-trivial
identity only for the modular case. This probably corresponds to a similar situation studied in
physics models. It would be interesting to uncover the connection and find more applications in
physics.

\section{$q$-Determinants and $q$-Pfaffians}
\subsection{$q$-Determinants}

Let $q$ be a complex non-zero number.
Let $\Lambda_n$ be the quantum exterior algebra  $F[x_1, \cdots,x_n]/I$, where $I$ is the ideal $<x_i^2, qx_ix_j+x_jx_i (i<j)>$. For
simplicity we still use the same symbol $x_i$ for the quotient $x_i+I$. Thus the wedge products satisfy the
following relations
\begin{align}\label{qwedge1}
&x_j \wedge x_i=(-q)x_i\wedge x_j, \\
&x_i\wedge x_i=0, \label{qwedge2}
\end{align}
where $i<j$. The algebra $\Lambda$ has a natural $\mathbb Z_{n+1}$-grading and decomposes itself:
\begin{equation*}
\Lambda=\bigoplus_{k=0}^n\Lambda_k,
\end{equation*}
where the $k$th homogeneous subspace $\Lambda_k$ is spanned by $x_{i_1}\cdots x_{i_k}$,
$1\leq i_1<\cdots<i_k\leq n$. It is clear that $dim(\Lambda)=2^n$.

Let $A=Mat_q(n)$ be the complex associative algebra generated by $a_{ij}$, $1\leq i, j\leq n$ subject to the relations:
\begin{align}\label{relation a1}
&a_{ik}a_{il}=qa_{il}a_{ik}, \\\label{relation a2}
&a_{ik}a_{jk}=qa_{jk}a_{ik}, \\\label{relation a3}
&a_{jk}a_{il}=a_{il}a_{jk}, \\\label{relation a4}
&a_{ik}a_{jl}-a_{jl}a_{ik}=(q-q^{-1})a_{il}a_{jk},
\end{align}
where $i<j$ and $k<l$.

The tensor product $A\otimes \Lambda$ inherits the grading of $\Lambda$:
\begin{equation*}
A\otimes \Lambda=\sum_{i=0}^n A\otimes \Lambda^i.
\end{equation*}
For convenience the general element $a\otimes x$ of $A\otimes \Lambda$ will be simply written as $ax$ and the multiplication of $A\otimes \Lambda$ is written as
\begin{equation*}
(ax)\wedge (by)=ab(x\wedge y).
\end{equation*}
Thus $x_ka_{ij}=a_{ij}x_k$ for any admissible $i, j, k$.

For each $i=1,\ldots, n$ we
consider the following special linear element (or {\it 1-form}) $\omega_i$ in $A\otimes \Lambda$
\begin{equation*}
\omega_i=\sum_{j=1}^n a_{ij}x_j.
\end{equation*}
It follows that the $\omega_i$ satisfies the same relation of the quantum exterior algebra.
\begin{align}
&\omega_i \wedge \omega_i=0, \nonumber\\
&\omega_j \wedge \omega_i=-q\omega_i \wedge w_j,\nonumber
\end{align}
where $i<j$.

The quantum determinant of $A$ is defined as follows.
\begin{equation}
{\det}_q(A)=\sum_{\sigma\in S_n}(-q)^{l(\sigma)}a_{1,\sigma(1)}\cdots a_{n,\sigma(n)}.
\end{equation}
The following result is well-known \cite{TT, KS}.
\begin{proposition} The quantum determinant ${\det}_q(A)$ satisfies the following relation:
\begin{equation}
\omega_1 \wedge \omega_2\wedge \cdots \wedge \omega_n={\det}_q(a_{ij})x_1 \wedge x_2 \wedge \cdots \wedge x_n.
\end{equation}
\end{proposition}
\begin{proof} We give a simpler proof for this well-knwon result.
Using $x_i\wedge x_i=0$ we expand $\omega_1\wedge\cdots\wedge\omega_n$ into a summation
of monomials $x_{i_1}\cdots x_{i_n}$ of distinct factors with coefficients from $A$:
\begin{align*}
\omega_1\wedge\omega_2\wedge \cdots \wedge\omega_n&=\sum_{i_1i_2\cdots i_n}a_{1i_1}a_{i_2}\cdots a_{ni_n}
x_{i_1}\wedge x_{i_2}\wedge \cdots \wedge x_{i_n}\\
&=(\sum_{\sigma\in S_n}(-q)^{l(\sigma)}a_{1,\sigma(1)}\cdots a_{n,\sigma(n)})x_1\wedge\cdots\wedge x_n,
\end{align*}
where we have reordered the monomial $x_{i_1}\wedge \cdots \wedge x_{i_n}$
using Eq. (\ref{qwedge1}).
\end{proof}

It follows from relation (\ref{relation a3}) that
\begin{equation*}
a_{\sigma(1),1}\cdots a_{\sigma(n),n}=a_{1,\sigma^{-1}(1)}\cdots a_{n,\sigma^{-1}(n)}
\end{equation*}
in algebra $A$. Since $l(\sigma^{-1})=l(\sigma)$, we then have
\begin{equation}
{\det}_q(A)=\sum_{\sigma\in S_n}(-q)^{l(\sigma)}a_{\sigma(1),1}\cdots a_{\sigma(n),n}.
\end{equation}

Let $I$ and $J$ be two subsets of $\{1,2,\cdots,n\}$ with $|I|=|J|=r$ with ordered
elements, i.e. $i_1<i_2<\cdots<i_r\in I$ and $j_1<j_2<\cdots<j_r\in J$. The quantum $r$-minor determinants are defined as \cite{NYM}
\begin{equation}
\xi^{i_1,\cdots,i_r}_{j_1,\cdots,j_r}=\sum_{\sigma\in S_r}(-q)^{l(\sigma)}a_{i_1,j_{\sigma(1)}}\cdots a_{i_r,j_{\sigma(r)}}.
\end{equation}
By relation (\ref{relation a3}) we know that
\begin{equation*}a_{i_{\sigma(1)},j_1}\cdots a_{i_{\sigma(n)},j_n}=a_{i_1,j_{\sigma^{-1}(1)}}\cdots a_{i_n,j_{\sigma^{-1}(n)}},
\end{equation*}
so we also have
\begin{equation}
\xi^{i_1,\cdots,i_r}_{j_1,\cdots,j_r}=\sum_{\sigma\in S_r}(-q)^{l(\sigma)}a_{i_{\sigma(1)},j_1}\cdots a_{i_{\sigma(r)},j_r}.
\end{equation}

\subsection{Laplace expansions of $q$-determinants and Pl\"ucker relations.}
 As the classical case the quantum determinant can also be expanded into a
 summation of quantum minors and the complement minor. The Laplace expansion
 is also treated by the exterior products. We include this for completeness.
\begin{proposition} For any $1\leq i_1,i_2,\cdots,i_n\leq n$, one has
\begin{equation*}\label{c1}
\sum_{\sigma\in S_n}(-q)^{l(\sigma)}a_{i_1,\sigma(1)}\cdots a_{i_n,\sigma(n)}=
\left\{ \begin{aligned}
&0,\ &if\ two\ i's\ coinside\\
&(-q)^{l(\pi)}{\det}_q(a_{ij})\ &if\ i's\ are\ distinct\\
\end{aligned} \right.
\end{equation*}

\begin{equation*}\label{c2}
\sum_{\sigma\in S_n}(-q)^{l(\sigma)}a_{\sigma(1),i_1}\cdots a_{\sigma(n),i_n}=
\left\{ \begin{aligned}
&0,\ &if\ two\ i's\ coinside\\
&(-q)^{l(\pi)}{\det}_q(a_{ij})\ &if\ i's\ are\ distinct\\
\end{aligned} \right.
\end{equation*}
where $\pi=\begin{pmatrix} 1 &2 &\cdots &n\\
i_1 &i_2 &\cdots &i_n\end{pmatrix}$.
\end{proposition}
\begin{proof}It is clear that $\omega_{i_1} \wedge \omega_{i_2}\wedge
\cdots \wedge \omega_{i_n}=0$
 whenever two indices coincide.
 For any permutation $\pi$ we have
 $$\omega_{{\pi}_1} \wedge \omega_{\pi_2}\wedge \cdots \wedge \omega_{\pi_n}=(-q)^{l(\pi)}\omega_{1} \wedge \omega_{2}\wedge \cdots \wedge \omega_{n}.$$
For any composition $(i_1i_2\cdots i_n)$ we can compute that
\begin{align*}
&\omega_{{i}_1} \wedge \omega_{i_2}\wedge \cdots \wedge \omega_{i_n}\\
&=\sum_{\sigma\in S_r}(-q)^{l(\sigma)}a_{i_1, \sigma(1)}\cdots a_{i_n, \sigma(n)}.
\end{align*}
Thus the first part of proposition is proved. The second part is treated similarly
by considering $\omega_i'=\sum_{j}a_{ji}x_j$. It is easy to see that for $i<j$
$$\omega_i' \wedge \omega_i'=0$$
$$\omega_j' \wedge \omega_i'=-q\omega_i' \wedge \omega_j'.$$
\begin{align*}
\omega_1' \wedge \omega_2'\wedge \cdots \wedge \omega_n'
=&\sum_{\sigma\in S_n}(-q)^{l(\sigma)}x_{\sigma(1),1}\cdots x_{\sigma(n),n}x_1 \wedge x_2 \wedge \cdots \wedge x_n\nonumber\\
=&{\det}_q(a_{ij})_{1\leq i,j\leq n}x_1 \wedge x_2 \wedge \cdots \wedge x_n.
\end{align*}
\end{proof}

Now we discuss the Laplace expansion of quantum determinant. We first choose r indices $i_1<i_2<\cdots<i_r$ from $1,2,\cdots,n$ and
let the remaining ones be ${i_{r+1} < i_{r+2} < \cdots < i_n}$. We have
\begin{equation}\label{laplace11}
\omega_{i_1} \wedge \omega_{i_2}\wedge \cdots \wedge \omega_{i_n}=(-q)^{i_1+\cdots+i_r-\frac{r(r+1)}{2}}\omega_1 \wedge \omega_2\wedge \cdots \wedge \omega_n
\end{equation}
since $\omega_j \wedge \omega_i=(-q)\omega_i \wedge \omega_j$, if $i<j$.
On the other hand,
\begin {equation}\label{laplace12}
\begin{split}
&\omega_{i_1} \wedge \omega_{i_2}\wedge \cdots \wedge \omega_{i_n}\\
=&(\omega_{i_1}\wedge \cdots \wedge w_{i_r})\wedge(w_{i_{r+1}}\wedge \cdots \wedge \omega_{i_n})\\
=&\sum_{j_1<\cdots<j_r}\xi_{j_1\cdots j_r}^{i_1 \cdots i_r}x_{j_1}\wedge\cdots\wedge x_{j_r}\wedge
\sum_{j_{r+1}<\cdots<j_n}\xi_{j_{r+1}\cdots j_n}^{i_{r+1} \cdots i_n}x_{j_{r+1}}\wedge\cdots\wedge x_{j_{n}}\\
=&\sum_{j_1<\cdots<j_r \atop j_{r+1}<\cdots<j_n}(-q)^{{j_1+\cdots+j_r-\frac{r(r+1)}{2}}}
\xi_{j_1\cdots j_r}^{i_1\cdots i_r}\xi_{j_{r+1}\cdots j_n}^{i_{r+1}\cdots i_n}x_{1}\wedge x_{2}\wedge\cdots \wedge x_{n},
\end{split}
\end{equation}
where we have used (\ref{laplace11}). Comparing (\ref{laplace11}) and (\ref{laplace12}), we get that
\begin{equation*}
{\det}_q(a_{ij})=\sum_{j_1<\cdots<j_r, j_{r+1}<\cdots<j_n}(-q)^{(j_1+\cdots+j_r)-
(i_1+\cdots+i_r)}\xi_{j_1 \cdots j_r}^{i_1 \cdots i_r}\xi_{j_{r+1} \cdots j_n}^{i_{r+1}\cdots i_n}.
\end{equation*}

A special case ($r=1$) of the Laplace expansion of ${\det}_q(A)$ is the quantum Cramer rule:
$${\det}_q(A)=\sum_{j=1}^n(-q)^{j-i}a_{ij}\Delta_q(ij),\qquad i=1,2,\cdots,n $$
where the cofactor $\Delta_q(ij)$ is the $q$-minor of size $(n-1)\times (n-1)$ obtained by deleting
the $i$th row and $j$th column from A. The orthogonality relations now read
$$\delta_{ik}{\det}_q(A)=\sum_{j=1}^n(-q)^{j-i}a_{ij}\Delta_q(kj),i=1,2,\cdots,n$$
where $\delta_{ij}$ is the Kronecker $\delta$ symbol.

One can also work with columns. Consider other $1$-forms by summing over the rows:
$$\omega_i'=\Sigma_{j=1}^n a_{ji}x_j.$$

We choose r indices $i_1<i_2<\cdots<i_r$ from $1,2,\cdots,n$ and let the
remaining ones be ${i_{r+1} < i_{r+2} < \cdots < i_n}$. Then we have

\begin{equation}\label{laplace21}
\omega_{i_1}' \wedge \omega_{i_2}'\wedge \cdots \wedge \omega_{i_n}'
=(-q)^{i_1+i_2+\cdots+i_r-\frac{r(r+1)}{2}}\omega_1' \wedge \omega_2'\wedge \cdots \wedge \omega_n'.
\end{equation}
since $\omega_j' \wedge \omega_i'=(-q)\omega_i' \wedge \omega_j'$, if $i<j$.
On the other hand,
\begin {equation}\label{laplace22}
\begin{split}
&\omega_{i_1}' \wedge \omega_{i_2}'\wedge \cdots \wedge \omega_{i_n}'\\
=&(\omega_{i_1}'\wedge \omega_{i_2}'\wedge \cdots \wedge \omega_{i_r}')\wedge
(\omega_{i_{r+1}}'\wedge \cdots \wedge \omega_{i_n}')\\
=&\sum_{j_1<\cdots<j_r\atop j_{r+1}<\cdots<j_{n}}
\xi^{j_1\cdots j_r}_{i_1 \cdots i_r}
x_{j_1}\wedge\cdots\wedge x_{j_r}\wedge
\xi^{j_{r+1}\cdots j_n}_{i_{r+1}\cdots i_n}
x_{j_{r+1}}\wedge\cdots\wedge x_{j_{n}}\\
=&\sum_{j_1<\cdots<j_r\atop j_{r+1}<\cdots<j_{n}}(-q)^{{j_1+j_2+\cdots+j_r-\frac{r(r+1)}{2}}}
\xi^{j_1,\cdots,j_r}_{i_1,\cdots,i_r}\xi^{j_{r+1}\cdots j_n}_{i_{r+1}\cdots i_n}
x_{1}\wedge x_{2}\wedge\cdots \wedge x_{n}.
\end{split}
\end{equation}

Comparing Eqs. (\ref{laplace21}) and (\ref{laplace22}), we get that
\begin{equation*}
{\det}_q(a_{ij})=\sum_{j_1<\cdots<j_r\atop j_{r+1}<\cdots<j_{n}}(-q)^{(j_1+\cdots+j_r)-(i_1+\cdots+i_r)}
\xi^{j_1\cdots j_r}_{i_1\cdots i_r}\xi^{j_{r+1}\cdots j_n}_{i_{r+1}\cdots i_n},
\end{equation*}
which is a $q$-analogous Laplace expansion of ${\det}_q(A)$. A special case of this
Laplace expansion is
$${\det}_q(A)=\sum_{j=1}^n(-q)^{j-i}a_{ji}\Delta_q(ji),i=1,2,\cdots,n.$$
Consequently we have the following orthogonal relations
$$\delta_{ik}{\det}_q(A)=\sum_{j=1}^n(-q)^{j-i}a_{ji}\Delta_q(jk),i=1,2,\cdots,n.$$

We choose $2n$ indices from $1,2,\cdots,2n$ such that $i_1<i_2<\cdots<i_n$ and $i_{n+1}<i_{n+2}<\cdots<i_{2n}.$
Let
$$\sigma=
\begin{pmatrix}
1  &2     &\cdots &2n\\
i_1 & i_2  &\cdots &i_{2n}\end{pmatrix},
\ \ \ \
\sigma_1=
\begin{pmatrix}
1       &\cdots  &n      &n+1 \cdots &2n\\
i_{n+1} &\cdots  &i_{2n} &i_1 \cdots &i_n
\end{pmatrix}.
$$

Then $l(\sigma_1)=n^2-l(\sigma)$, and  we have the following result.

\begin{proposition}\label{thp1}
Let $\sigma=\begin{pmatrix}
1  &2     &\cdots &2n\\
i_1 & i_2  &\cdots &i_{2n}\end{pmatrix}
$ be a permutation in $S_{2n}$ such
that $i_k=k$ for $1\leq k\leq r$, $i_{r+1}<i_{r+2}<\cdots<i_{n}$  and  $i_{n+1}<i_{n+2}<\cdots<i_{2n}$. Then
\begin{equation}
\begin{split}
&\sum_{i_{r+1}<\cdots<i_{n},i_{n+1}<\cdots<i_{2n}}(-q)^{l(\sigma)}\xi^{1,\cdots,n}_{i_1,\cdots,i_n}\xi^{1,\cdots,n}_{i_{n+1},\cdots,i_{2n}}=0,\\
\end{split}
\end{equation}
\begin{equation}
\begin{split}
&\sum_{i_{r+1}<\cdots<i_{n},i_{n+1}<\cdots<i_{2n}}(-q)^{-l(\sigma)}
\xi^{1,\cdots,n}_{i_{n+1},\cdots,i_{2n}}\xi^{1,\cdots,n}_{i_1,\cdots,i_n}=0.
\end{split}
\end{equation}
\end{proposition}

\begin{proof}Let $\omega_{i}=\sum_{k=1}^{2n}a_{ik}x_k,\ \omega_{ir}=\sum_{k=1}^{r}a_{ik}x_k,\ \omega_{ir}'=\sum_{k=r+1}^{2n}a_{ik}x_k$,
it is easy to see that for any $i<j$
\begin{align*}
&\omega_{jr}\wedge\omega_{ir}=(-q)\omega_{ir}\wedge\omega_{jr},\\
&\omega_{jr}'\wedge\omega_{ir}'=(-q)\omega_{ir}'\wedge\omega_{jr}',\\
&\omega_{ir}'\wedge\omega_{jr}=(-q)\omega_{jr}\wedge\omega_{ir}',\\
&\omega_{i}\wedge\omega_{jr}=\omega_{jr}\wedge((-q)^{-1}\omega_{ir}+(-q)\omega_{ir}').
\end{align*}
Note that $r<n$, so we compute that
\begin{equation}\label{p11}
\begin{split}
&\omega_1\wedge\omega_2\wedge\cdots\wedge\omega_n\wedge\omega_{1r}'\wedge\cdots\wedge\omega_{nr}'\\
=&\omega_1\wedge\omega_2\wedge\cdots\wedge\omega_{n-1}\wedge(\omega_{nr}+\omega_{nr}')\wedge\omega_{1r}'\wedge\cdots\wedge\omega_{nr}'\\
=&\omega_1\wedge\omega_2\wedge\cdots\wedge\omega_{n-1}\wedge\omega_{nr}\wedge\omega_{1r}'\wedge\cdots\wedge\omega_{nr}'\\
=&\omega_{nr}\wedge((-q)^{-1}\omega_{1r}+(-q)\omega_{1r}')\wedge\cdots \\
& \qquad\qquad \wedge((-q)^{-1}\omega_{n-1,r}+(-q)\omega_{n-1,r}')\wedge\omega_{1r}'\wedge\cdots\wedge\omega_{nr}'\\
=&\omega_{nr}\wedge(-q)^{-1}\omega_{n-1,r}\wedge\cdots\wedge(-q)^{n-1}\omega_{1,r}\wedge\omega_1'\wedge\cdots\wedge\omega_n'\\
=&0,\\
\end{split}
\end{equation}
Similarly, we can get that
\begin{equation}\label{p12}
\begin{split}
\omega_{1r}'\wedge\cdots\wedge\omega_{nr}'\wedge\omega_1\wedge\omega_2\wedge\cdots\wedge\omega_n=0.\\
\end{split}
\end{equation}

On the other hand,
\begin{equation}\label{p13}
\begin{split}
&\omega_1\wedge\omega_2\wedge\cdots\wedge\omega_n\wedge\omega_{1r}'\wedge\cdots\wedge\omega_{nr}'\\
=&\sum_{i_{r+1}<\cdots<i_{n}\atop i_{n+1}<\cdots<i_{2n}}(-q)^{l(\sigma)}\xi^{1,\cdots,n}_{i_1,\cdots,i_n}\xi^{1,\cdots,n}_{i_{n+1},\cdots,i_{2n}}x_1\wedge x_2\wedge\cdots\wedge x_n,\\
\end{split}
\end{equation}
\begin{equation}\label{p14}
\begin{split}
&\omega_{1r}'\wedge\cdots\wedge\omega_{nr}'\wedge\omega_1\wedge\omega_2\wedge\cdots\wedge\omega_n\\
=&\sum_{i_{r+1}<\cdots<i_{n}\atop i_{n+1}<\cdots<i_{2n}}(-q)^{n^2-l(\sigma)}\xi^{1,\cdots,n}_{i_{n+1},\cdots,i_{2n}}\xi^{1,\cdots,n}_{i_1,\cdots,i_n}x_1\wedge x_2\wedge\cdots\wedge x_n.\\
\end{split}
\end{equation}

Comparing Eqs. (\ref{p11}), (\ref{p12}), (\ref{p13}) and (\ref{p14}), we obtain the proposition.
\end{proof}

\begin{corollary}\label{thp2}
Let $1=i_1<i_2<\cdots<i_n$ be n indices between $1$ and $2n$, and the remaining indices
 are $i_{n+1}<i_{n+2}<\cdots<i_{2n}$. One then has
\begin{equation}
\sum_{1=i_1<\cdots<i_n,i_{n+1}<\cdots<i_{2n}}(-q)^{l(\sigma)}
\xi^{1,\cdots,n}_{i_1,\cdots,i_n}\xi^{1,\cdots,n}_{i_{n+1},\cdots,i_{2n}}=0,
\end{equation}
\begin{equation}
\sum_{1=i_1<\cdots<i_n,i_{n+1}<\cdots<i_{2n}}(-q)^{-l(\sigma)}
\xi^{1,\cdots,n}_{i_{n+1},\cdots,i_{2n}}\xi^{1,\cdots,n}_{i_1,\cdots,i_n}=0,
\end{equation}
where
$\sigma=\begin{pmatrix}
1 &2    &\cdots  & 2n\\
i_1  &i_2  &\cdots  & i_{2n}\end{pmatrix}.$

Denote $\xi^{1,\cdots,n}_{i_1,\cdots,i_n}$ by $(i_1,i_2,\cdots,i_n)$, the results can be rewritten as
\begin{equation}\nonumber
\begin{split}
&\sum_{1=i_1<\cdots<i_n,i_{n+1}<\cdots<i_{2n}}(-q)^{l(\sigma)}(i_1,\cdots,i_n)(i_{n+1},\cdots,i_{2n})=0,\\
&\sum_{1=i_1<\cdots<i_n,i_{n+1}<\cdots<i_{2n}}(-q)^{-l(\sigma)}(i_{n+1},\cdots,i_{2n})(i_1,\cdots,i_n)=0.
\end{split}
\end{equation}

In particular, if $n=2$, we have
$$(1,2)(3,4)+(-q)(1,3)(2,4)+(-q)^2(1,4)(2,3)=0, $$
which is the simplest {\it Pl\"{u}cker\ relation}.
\end{corollary}

\begin{proposition}\label{thp3}
Let $\sigma=\begin{pmatrix}
1  &2     &\cdots &2n\\
i_1 & i_2  &\cdots &i_{2n}\end{pmatrix}
$ be a permutation in $S_{2n}$ such
that $i_k=k$ for $1\leq k\leq r$, $i_{r+1}<i_{r+2}<\cdots<i_{n}$  and  $i_{n+1}<i_{n+2}<\cdots<i_{2n}$. Then
\begin{equation*}
\begin{split}
&\sum_{i_{r+1}<\cdots<i_{n},i_{n+1}<\cdots<i_{2n}}(-q)^{l(\sigma)}
\xi^{n+1,\cdots,2n}_{i_1,\cdots,i_n}\xi^{1,\cdots,n}_{i_{n+1},\cdots,i_{2n}}\\
=&(-q)^{(n^2-2nr)}
\sum_{i_{r+1}<\cdots<i_{n},i_{n+1}<\cdots<i_{2n}}(-q)^{n^2-l(\sigma)}
\xi^{1,\cdots,n}_{i_{n+1},\cdots,i_{2n}}\xi^{n+1,\cdots,2n}_{i_1,\cdots,i_n}.
\end{split}
\end{equation*}
\end{proposition}

\begin{proof}
It is clear that for any $i<j$
$$\omega_j\wedge\omega_{ir}'=\omega_{ir}'\wedge((-q)^{-1}\omega_{jr}+(-q)\omega_{jr}').$$
So we have
\begin{equation}\label{p21}
\begin{split}
&\omega_{n+1}\wedge \omega_{n+2}\wedge\cdots\wedge \omega_{2n}\wedge \omega_{1r}'\wedge \omega_{2r}'\wedge\cdots\wedge \omega_{nr}'\\
=&\omega_{1r}'\wedge \omega_{2r}'\wedge\cdots\wedge \omega_{nr}'\wedge((-q)^{-n}\omega_{n+1,r}+(-q)^{n}\omega_{n+1,r}')\\
&\qquad \qquad \wedge\cdots\wedge((-q)^{-n}\omega_{2n,r}+(-q)^{n}\omega_{2n,r}')\\
=&(-q)^{n^2}\omega_{1r}'\wedge \omega_{2r}'\wedge\cdots\wedge \omega_{nr}'\wedge((-q)^{-2n}\omega_{n+1,r}+\omega_{n+1,r}')\wedge\cdots \\ &\qquad\qquad \wedge((-q)^{-2n}\omega_{2n,r}+\omega_{2n,r}')\\
=&(-q)^{n^2-2nr}\omega_{1r}'\wedge \omega_{2r}'\wedge\cdots\wedge \omega_{nr}'\wedge \omega_{n+1}\wedge\cdots\wedge \omega_{2n}.
\end{split}
\end{equation}
On the other hand, we can expand the wedge product of $\omega's$ as follows.
\begin{equation}\label{p22}
\begin{split}
&\omega_{n+1}\wedge \omega_{n+2}\wedge\cdots\wedge \omega_{2n}\wedge \omega_1'\wedge \omega_2'\wedge\cdots\wedge \omega_n'\\
=&\sum_{1=i_1<i_2<\cdots<i_n}(-q)^{l(\sigma)}
\xi^{n+1,\cdots,2n}_{i_1,\cdots,i_n}\xi^{1,\cdots,n}_{i_{n+1},\cdots,i_{2n}}
x_1\wedge x_2\wedge\cdots\wedge x_{2n},\\
\end{split}
\end{equation}

\begin{equation}\label{p23}
\begin{split}
&\omega_1'\wedge \omega_2'\wedge\cdots\wedge \omega_n'\wedge \omega_{n+1}\wedge\cdots\wedge \omega_{2n}\\
=&\sum_{1=i_1<i_2<\cdots<i_n}(-q)^{n^2-l(\sigma)}
\xi^{1,\cdots,n}_{i_{n+1},\cdots,i_{2n}}\xi^{n+1,\cdots,2n}_{i_1,\cdots,i_n}
x_1\wedge x_2\wedge\cdots\wedge x_{2n}.\\
\end{split}
\end{equation}

Comparing the equations (\ref{p21}), (\ref{p22}) and (\ref{p23}), we obtain the proposition.
\end{proof}

\begin{corollary}\label{thp4}
Let $1=i_1<i_2<\cdots<i_n$ be n indices between $1$ and $2n$, and the remaining
 indices are $i_{n+1}<i_{n+2}<\cdots<i_{2n}$. Then one has
\begin{equation*}
\begin{split}
&\sum_{1=i_1<i_2<\cdots<i_{n},i_{n+1}<\cdots<i_{2n}}(-q)^{l(\sigma)}
\xi^{n+1,\cdots,2n}_{i_1,\cdots,i_n}\xi^{1,\cdots,n}_{i_{n+1},\cdots,i_{2n}}\\
=&(-q)^{(n^2-2n)}
\sum_{1=i_1<i_{2}<\cdots<i_{n},i_{n+1}<\cdots<i_{2n}}(-q)^{n^2-l(\sigma)}
\xi^{1,\cdots,n}_{i_{n+1},\cdots,i_{2n}}\xi^{n+1,\cdots,2n}_{i_1,\cdots,i_n},\\
\end{split}
\end{equation*}
where
$\sigma=\begin{pmatrix}
1 &2    &\cdots  & 2n\\
i_1  &i_2  &\cdots  & i_{2n}\end{pmatrix}.
$

In particular, if $n=2$,
\begin{equation*}
\sum_{1=i_1<i_2,i_3<i_4}(-q)^{l(\sigma)}
\xi^{3,4}_{i_1,i_2}\xi^{1,2}_{i_{3},\cdots,i_{4}}
=\sum_{1=i_1<i_2,i_3<i_4}(-q)^{n^2-l(\sigma)}
\xi^{1,2}_{i_1,i_2}\xi^{3,4}_{i_{3},,i_{4}}.
\end{equation*}
\end{corollary}

\begin{remark}\label{remark1} Due to the symmetry of the relations in the algebra $A=M_q(2n)$
Proposition (\ref{thp1}), Corollary (\ref{thp2}), Proposition (\ref{thp3}) and Corollary (\ref{thp4})
also hold if we replace $\xi^I_J$ by $\xi^J_I$. Moreover,
if we substitute $1,2,\cdots,2n$ by $j_1<j_2<\cdots<j_{2n}$, similar identities also hold.
\end{remark}

\subsection{$q$-Pfaffians.} Let $B$ be the algebra generated by
$b_{ij}$ for $1\leq i<j\leq n$ modulo the ideal generated by the
relations
\begin{equation}\label{reb}
b_{ij}b_{kl}+(-q)b_{ik}b_{jl}+(-q)^{2}b_{il}b_{jk}
=b_{kl}b_{ij}+(-q)^{-1}b_{jl}b_{ik}+(-q)^{-2}b_{jk}b_{il},
\end{equation}
where $i<j<k<l$.

We define the $q$-Pfaffian\ $[1,2,\cdots ,2n]$ inductively as follows. For $i<j$, let $[i,j]=b_{ij}$.
Then the $n$th order $q$-Pfaffian is inductively defined to be
$$\mbox{Pf}_q=[1,2,\cdots ,2n]=\sum_{j=2}^{2n}(-q)^{j-2}[1,j][2,3,\cdots,\hat{j},\cdots,2n].$$

To obtain an explicit formula for $\mbox{Pf}_q$, we consider some subsets of permutations.
Let $N=\{1,2,\cdots,2n\}$, $S$ be a subset of $N$ and $|S|=2p$, we define two sets of partitions of
$S$ into unordered (resp. unordered) pairs respectively:
\begin{align*}
&\Pi(S)=\{(i_1,j_1), (i_2,j_2), \cdots, (i_{p},j_{p});i_k<j_k,\ i_k<i_{k+1},\ and\ i_k,j_k\in S\},\\
&\Pi(S)'=\{(i_1,j_1), (i_2,j_2), \cdots, (i_{p},j_{p});i_k<j_k\,\ and\ i_k,j_k\in S\}.
\end{align*}
Note that $\Pi(S)\subset \Pi(S)'$, and we also denote $\Pi=\Pi(N)$, $\Pi'=\Pi(N)'$.

We can associate the element of $\Pi(S)'$ with the
subset of the symmetric group $S_{2n}$ in the following manner
$$\sigma=
\begin{pmatrix}
1  &2   &3    &4   &\cdots &2p-1 &2p\\
i_1 &j_1 & i_2 &j_2 &\cdots &i_p &j_p
\end{pmatrix}\in S_{2p}
$$
for $\pi=\{(i_1,j_1), (i_2,j_2), \cdots, (i_p,j_p)\}\in\Pi(S)'$.  We define
$l(\pi)=l(\sigma)$, and denote
$b_\pi=b_{i_1,j_1}b_{i_2,j_2}\cdots b_{i_p,j_p}$.

Let $\pi={\{(i_1,j_1), \cdots, (i_n,j_n)\}}\in\Pi'$, $\pi_1=\{(i_1,j_1), \cdots(i_p,j_p)\}\in\Pi(S)'$,
$\pi_2=\{(i_{p+1},j_{p+1}), \cdots, (i_n,j_n)\}\in\Pi(N\backslash S)'$,
it is clear that $$l(\pi)=l(\pi_1)+l(\pi_2)+\sum_{i\in S}i-p(2p+1).$$

Here is an expression of the quantum Pfaffian that works for any $q$.
\begin{theorem}\label{T:Ph1} One has that
 $$[1,2,\cdots ,2n]=\sum_{\pi\in\Pi}(-q)^{l(\pi)}b_\pi=\sum_{\pi \in\Pi}(-q)^{l(\pi)}[i_1,j_1][i_2,j_2]\cdots[i_n,j_n].$$
\end{theorem}
\begin{proof}
We use induction. The initial step $n=1$ is trivial. Suppose the formula is true for $n-1$, then
\begin{align*}
&\sum_{j=2}^{2n}(-q)^{j-2}b_{1,j}[2,3,\cdots,\hat{j},\cdots,2n]\\
=&\sum_{j=2}^{2n}(-q)^{j-2}b_{1,j}\sum_{\pi\in\Pi(N\backslash\{1,j\})}(-q)^{l(\pi)}b_{\pi}\\
=&\sum_{j=2}^{2n}\sum_{\pi\in\Pi(N\backslash\{1,j\})}(-q)^{l(\pi)+j-2}b_{1,j}b_{\pi}\\
=&\sum_{\pi \in \Pi}(-q)^{l(\pi)}b_{\pi}\\
=&\sum_{\pi \in \Pi}(-q)^{l(\pi)}[i_1,j_1][i_2,j_2]\cdots[i_n,j_n].
\end{align*}
\end{proof}

\begin{lemma}\label{L:Ph2a} We have that
$$\sum_{i<j}(-q)^{i+j-3}[i,j][1,2,\cdots,\hat{i},\cdots,\hat{j},\cdots,2n]=
(\sum_{i=0}^{n-1}q^{4i})[1,2,\cdots ,2n].$$
\end{lemma}
\begin{proof}
We prove the statement by induction on $n$.
when $n=2$ it is obviously true. Suppose it is true for $n-1$, then

\begin{align*}
&\sum_{i<j}(-q)^{i+j-3}[i,j][1,2,\cdots,\hat{i},\cdots,\hat{j},\cdots,2n]\\
=&\sum_{1<j}(-q)^{j-2}[1,j][2,3,\cdots,\hat{j},\cdots,2n]\\
+&\sum_{1<i<j}(-q)^{i+j-3}[i,j][1,2,\cdots,\hat{i},\cdots,\hat{j},\cdots,2n].\\
\end{align*}

It follows from definition that
\begin{equation}
\begin{split}
&\sum_{1<j}(-q)^{j-2}[1,j][2,3,\cdots,\hat{j},\cdots,2n]=[1,2,\cdots ,2n].\\
\end{split}
\end{equation}

As for the second summand we have that
\begin{align*}
&\sum_{1<i<j}(-q)^{i+j-3}[i,j][1,2,\cdots,\hat{i},\cdots,\hat{j},\cdots,2n]\\
=&\sum_{1<i<j}\sum_{k\neq 1,i,j}(-q)^{i+j+k-9+l((i,j)(1,k))}b_{i,j}b_{1,k}\sum_{\pi\in N\backslash\{1,i,j,k\}}(-q)^{l(\pi)}b_{\pi}\\
=&\sum_{1<i<j<k}(-q)^{i+j+k-9}\sum_{\pi_1\in\{1,i,j,k\}}(-q)^{4-l(\pi_1)}b_{i_2,j_2}b_{i_1,j_1}\sum_{\pi_2\in N\backslash\{1,i,j,k\}}(-q)^{l(\pi_2)}b_{\pi_2}\\
=&\sum_{1<i<j<k}(-q)^{i+j+k-9}\sum_{\pi_1\in\{1,i,j,k\}}(-q)^{4+l(\pi_1)}b_{i_1,j_1}b_{i_2,j_2}\sum_{\pi_2\in N\backslash\{1,i,j,k\}}(-q)^{l(\pi_2)}b_{\pi_2}\\
=&(-q)^4\sum_{1<i<j<k}\sum_{\pi_1\in\{1,i,j,k\}}\sum_{\pi_2\in N\backslash\{1,i,j,k\}}(-q)^{l(\pi_1\pi_2)}b_{\pi_1}b_{\pi_2}\\
=&(-q)^4\sum_{i=2}^{2n}(-q)^{i-2}b_{1i}\sum_{j<k,j\neq i,k\neq i}\sum_{\pi\in N\backslash\{1,i,j,k\}}(-q)^{l((j,k)\pi)}b_{jk}b_{\pi}\\
=&(-q)^4\sum_{i=2}^{2n}(-q)^{i-2}[1,i](\sum_{i=0}^{n-2}(-q)^{4i})[2,3,\cdots,\hat{i},\cdots,2n]\\
=&(\sum_{i=1}^{n-1}(-q)^{4i})[1,2,\cdots ,2n].
\end{align*}
Therefore we have
$$\sum_{i<j}(-q)^{i+j-3}[i,j][1,2,\cdots,\hat{i},\cdots,\hat{j},\cdots,2n]=(\sum_{i=0}^{n-1}(-q)^{4i})[1,2,\cdots ,2n].$$
\end{proof}

\begin{theorem}\label{pf1} The quantum Pfaffian can be computed by
$$\sum_{\pi\in \Pi'}(-q)^{l(\pi)}b_{\pi}
=\sum_{\pi\in \Pi'}(-q)^{l(\pi)}[i_1,j_1][i_2,j_2]\cdots[i_n,j_n]
={\sum_{\sigma \in S_n} q^{4l(\sigma)}}[1,2,\cdots ,2n].$$
\end{theorem}
\begin{proof} This is proved again by induction on n. The case of
$n=2$ is obvious. Suppose it is true for $n-1$, then
\begin{align*}
&\sum_{\pi\in \Pi'}(-q)^{l(\pi)}[i_1,j_1][i_2,j_2]\cdots[i_n,j_n]\\
=&\sum_{i<j}\sum_{\pi\in\Pi(N\backslash\{i,j\})'}(-q)^{l(\pi)+i+j-3}[i,j]b_{\pi}\\
=&\sum_{i<j}(-q)^{i+j-3}[i,j]({\sum_{\sigma \in S_{n-1}} q^{4l(\sigma)}})[1,2,\cdots,\hat{i},\cdots,\hat{j},\cdots,2n]\\
=&\big(\sum_{\sigma\in S_{n-1}}q^{4l(\sigma)}\big)(\sum_{i=0}^{n-1}q^{4i})[1,2,\cdots ,2n]\\
=&{\sum_{\sigma \in S_n}q^{4l(\sigma)}}[1,2,\cdots ,2n],
\end{align*}
where we have used Lemma \ref{L:Ph2a} in the last step.
\end{proof}

The following result shows that the quantum Pfaffian is a volume element
for a quantum exterior form.

\begin{theorem}
Let $\Omega=\sum_{1\leq i<j\leq 2n}b_{ij}(x_i\wedge x_j)$, one has
$$\wedge ^n\Omega=({\sum_{\sigma \in S_n} q^{4l(\sigma)}})[1,2,\cdots ,2n]x_1\wedge x_2\wedge\cdots \wedge x_{2n}.$$
\end{theorem}

\begin{proof} Note that
$$\wedge ^n\Omega=\sum_{\pi\in \Pi'}(-q)^{l(\pi)}b_{\pi}\\x_1\wedge x_2\wedge\cdots \wedge x_{2n}.$$
The result follows immediately from Theorem \ref{pf1}.
\end{proof}

It is known that every $2n$-order determinant can be expressed as an $n$-order
Pfaffian. In fact for an anti-symmetric matrix $A=(a_{ij})_{1\leq i,j\leq 2n}$, let
\begin{equation*}
b_{ij}=\sum_{m=1}^n(a_{2m-1,i}a_{2m,j}-a_{2m-1,j}a_{2m,i}),1\leq
i<j\leq 2n,
\end{equation*}
then
$$\mbox{det}(A)=\mbox{Pf} (b_{ij}).
$$

Similarly, in the quantum coordinate ring $A$ if we let
$$b_{ij}=\sum_{m=1}^n(a_{2m-1,i}a_{2m,j}-qa_{2m-1,j}a_{2m,i}),1\leq i<j\leq 2n,$$
then the matrix $(b_{ij})$ also satisfies some
quadratic relations called quantum anti-symmetric relations
\cite{FRT, R, D}. It was proved \cite{JR} that $\mbox{Pf}_q(b_{ij})={\det}_q(a_{ij})$, which
was obtained by using representation theory of
the quantum enveloping algebra. Here we give a new proof
independent from representation theory.

\begin{theorem} On the quantum coordinate ring $M_q(2n)=<a_{ij}>$ one has that
$$\Pf_q(B)={\det}_q(A),$$
where $b_{ij}=\sum_{m=1}^n(a_{2m-1,i}a_{2m,j}-qa_{2m-1,j}a_{2m,i})$.
\end{theorem}
\begin{proof}
First we show that
$b_{ij}=\sum_{m=1}^n(a_{2m-1,i}a_{2m,j}-qa_{2m-1,j}a_{2m,i})$ satisfy the relation (\ref{reb}).
Note that
Proposition \ref{thp1} is true for any distinctive rows or columns. Therefore we have the following relations:
\begin{equation*}
\begin{split}
&\xi^{2m-1,2m}_{i,j}\xi^{2m-1,2m}_{k,l}-q\xi^{2m-1,2m}_{i,k}\xi^{2m-1,2m}_{j,l}+q^2\xi^{2m-1,2m}_{i,l}\xi^{2m-1,2m}_{j,k}=0,\\
&\xi^{2m-1,2m}_{k,l}\xi^{2m-1,2m}_{i,j}-q^{-1}\xi^{2m-1,2m}_{j,l}\xi^{2m-1,2m}_{i,k}+q^{-2}
\xi^{2m-1,2m}_{j,k}\xi^{2m-1,2m}_{i,l}=0.
\end{split}
\end{equation*}
where $m=1,2,\cdots,n$ and $i<j<k<l$.

Similarly Proposition \ref{thp3} is also independent from the rows and columns. If $n=2$, and we choose rows $2s-1, 2s, 2t-1, 2t$, where $s<t$, and columns
$i<j<k<l$, so we have the following relations.
\begin{align*}
&\xi^{2t-1,2t}_{i,j}\xi^{2s-1,2s}_{k,l}+(-q)\xi^{2t-1,2t}_{i,k}\xi^{2s-1,2s}_{j,l}+(-q)^2\xi^{2t-1,2t}_{i,l}\xi^{2s-1,2s}_{j,k}\\
=&(-q)^4\xi^{2s-1,2s}_{k,l}\xi^{2t-1,2t}_{i,j}+(-q)^{3}\xi^{2s-1,2s}_{j,l}\xi^{2t-1,2t}_{i,k}
+(-q)^2\xi^{2s-1,2s}_{j,k}\xi^{2t-1,2t}_{i,l}.
\end{align*}
So we have
\begin{align*}
&\xi^{2s-1,2s}_{i,j}\xi^{2t-1,2t}_{k,l}+(-q)\xi^{2s-1,2s}_{i,k}\xi^{2t-1,2t}_{j,l}+(-q)^2\xi^{2s-1,2s}_{i,l}\xi^{2t-1,2t}_{j,k}\\
&+\xi^{2t-1,2t}_{i,j}\xi^{2s-1,2s}_{k,l}+(-q)\xi^{2t-1,2t}_{i,k}\xi^{2s-1,2s}_{j,l}+(-q)^2\xi^{2t-1,2t}_{i,l}\xi^{2s-1,2s}_{j,k}\\
=&\xi^{2s-1,2s}_{i,j}\xi^{2t-1,2t}_{k,l}+(-q)\xi^{2s-1,2s}_{i,k}\xi^{2t-1,2t}_{j,l}+(-q)^2\xi^{2s-1,2s}_{i,l}\xi^{2t-1,2t}_{j,k}\\
&+(-q)^4\xi^{2s-1,2s}_{k,l}\xi^{2t-1,2t}_{i,j}+(-q)^{3}\xi^{2s-1,2s}_{j,l}\xi^{2t-1,2t}_{i,k}+(-q)^2\xi^{2s-1,2s}_{j,k}\xi^{2t-1,2t}_{i,l}\\
=&\xi^{2s-1,2s,2t-1,2t}_{i,j,k,l}.
\end{align*}
Also we have that
\begin{align*}
&\xi^{2t-1,2t}_{k,l}\xi^{2s-1,2s}_{i,j}+(-q)^{-1}\xi^{2t-1,2t}_{j,l}\xi^{2s-1,2s}_{i,k}+(-q)^{-2}\xi^{2t-1,2t}_{j,k}\xi^{2s-1,2s}_{i,l}\\
&+\xi^{2s-1,2s}_{k,l}\xi^{2t-1,2t}_{i,j}+(-q)^{-1}\xi^{2s-1,2s}_{j,l}\xi^{2t-1,2t}_{i,k}+(-q)^{-2}\xi^{2s-1,2s}_{j,k}\xi^{2t-1,2t}_{i,l}\\
=&\xi^{2t-1,2t}_{k,l}\xi^{2s-1,2s}_{i,j}+(-q)^{-1}\xi^{2t-1,2t}_{j,l}\xi^{2s-1,2s}_{i,k}+(-q)^{-2}\xi^{2t-1,2t}_{j,k}\xi^{2s-1,2s}_{i,l}\\
&+(-q)^{-4}\xi^{2t-1,2t}_{i,j}\xi^{2s-1,2s}_{k,l}+(-q)^{-3}\xi^{2t-1,2t}_{i,k}\xi^{2s-1,2s}_{j,l}+(-q)^{-2}\xi^{2t-1,2t}_{i,l}\xi^{2s-1,2s}_{j,k}\\
=&\xi^{2s-1,2s,2t-1,2t}_{i,j,k,l}.
\end{align*}

Since $b_{ij}$ can be expressed as
\begin{equation}\label{e:b}
b_{ij}=\sum_{m=1}^n(a_{2m-1,i}a_{2m,j}-qa_{2m-1,j}a_{2m-1,i})=\sum_{m=1}^n\xi^{2m-1,2m}_{i,j},
\end{equation}
then we have
\begin{align*}
&b_{ij}b_{kl}+(-q)b_{ik}b_{jl}+(-q)^{2}b_{il}b_{jk}\\
=&\sum_{m=1}^n\xi^{2m-1,2m}_{i,j}\sum_{m=1}^n\xi^{2m-1,2m}_{k,l}+(-q)\sum_{m=1}^n\xi^{2m-1,2m}_{i,k}\sum_{m=1}^n\xi^{2m-1,2m}_{j,l}\\
+&(-q)^2\sum_{m=1}^n\xi^{2m-1,2m}_{i,l}\sum_{m=1}^n\xi^{2m-1,2m}_{j,k}\\
=&\sum_{s<t}(\xi^{2s-1,2s}_{i,j}\xi^{2t-1,2t}_{k,l}+\xi^{2t-1,2t}_{i,j}\xi^{2s-1,2s}_{k,l})+(-q)\sum_{s<t}(\xi^{2s-1,2s}_{i,k}\xi^{2t-1,2t}_{j,l}+\xi^{2t-1,2t}_{i,k}\xi^{2s-1,2s}_{j,l})\\
&+(-q)^2\sum_{s<t}(\xi^{2s-1,2s}_{i,l}\xi^{2t-1,2t}_{j,k}+\xi^{2t-1,2t}_{i,l}\xi^{2s-1,2s}_{j,k})\\
=&\sum_{s<t}\xi^{2s-1,2s,2t-1,2t}_{i,j,k,l}.
\end{align*}
On the other hand we also have
\begin{align*}
&b_{kl}b_{ij}+(-q)^{-1}b_{jl}b_{ik}+(-q)^{-2}b_{jk}b_{il}\\
=&\sum_{m=1}^n\xi^{2m-1,2m}_{k,l}\sum_{m=1}^n\xi^{2m-1,2m}_{i,j}+(-q)^{-1}\sum_{m=1}^n\xi^{2m-1,2m}_{j,l}\sum_{m=1}^n\xi^{2m-1,2m}_{i,k}\\
&+(-q)^{-2}\sum_{m=1}^n\xi^{2m-1,2m}_{j,k}\sum_{m=1}^n\xi^{2m-1,2m}_{i,l}\\
=&\sum_{s<t}(\xi^{2t-1,2t}_{k,l}\xi^{2s-1,2s}_{i,j}+\xi^{2s-1,2s}_{k,l}\xi^{2t-1,2t}_{i,j})+(-q)^{-1}\sum_{s<t}(\xi^{2t-1,2t}_{j,l}\xi^{2s-1,2s}_{i,k}+\xi^{2s-1,2s}_{j,l}\xi^{2t-1,2t}_{i,k})\\
&+(-q)^{-2}\sum_{s<t}(\xi^{2t-1,2t}_{j,k}\xi^{2s-1,2s}_{i,l}+\xi^{2s-1,2s}_{j,k}\xi^{2t-1,2t}_{i,l})\\
=&\sum_{s<t}\xi^{2s-1,2s,2t-1,2t}_{i,j,k,l}.
\end{align*}

Therefore $b_{ij}$ satisfy the quadratic Pl\"ucker relations:
$$b_{ij}b_{kl}+(-q)b_{ik}b_{jl}+(-q)^{2}b_{il}b_{jk}=b_{kl}b_{ij}+(-q)^{-1}b_{jl}b_{ik}+(-q)^{-2}b_{jk}b_{il}.
$$
Now we consider the two-form
\begin{equation}
\Omega=\omega_1\wedge\omega_2+\omega_3\wedge\omega_4+\cdots+\omega_{2n-1}\wedge\omega_{2n},
\end{equation}
where $\omega_i=\Sigma_{j=1}^{2n} a_{ij}x_j,(i =1,2, . . . , 2n)$. Taking wedge products, we have
\begin{align} \nonumber  
\bigwedge^n\Omega&=(\sum_{\sigma\in S_n}q^{4l(\sigma)})\omega_1\wedge\omega_2\cdots\wedge\omega_{2n-1}\wedge\omega_{2n}\\ \label{p2}
&=(\sum_{\sigma\in S_n}q^{4l(\sigma)})\mbox{det}_q(a_{ij})x_1 \wedge x_2 \wedge \cdots \wedge x_{2n},
\end{align}
by definition of the quantum determinant.
On the other hand,
$$\Omega=\sum_{1\leq i<j\leq 2n}b_{ij}x_i\wedge x_j,$$
where $b_{ij}$ are given by Eq. (\ref{e:b}), $1\leq i<j\leq 2n$.
We have showed that $b_{ij}$ satisfy Eq.(\ref{reb}), so we have
\begin{equation}\label{om2}
\bigwedge^n\Omega=(\sum_{\sigma\in S_n}(-q)^{4l(\sigma)})[1,2,\cdots,2n]x_1 \wedge x_2\wedge\cdots\wedge x_n,
\end{equation}
where $[i,j]=b_{ij}$, $i<j$.
Comparing two equations (\ref{p2}) and (\ref{om2}) we obtain that
$$\mbox{det}_q(A)=[1,2,\cdots,2n].$$
\end{proof}

The following corollary shows that an $(2n+1)$th-order $q$-determinant can be expressed as a $(n+1)$th-order $q$-Pfaffian.

\begin{corollary} Any $(2n+1)$th-order $q$-determinant
can be expressed as a $(n+1)$th-order $q$-Pfaffian.
\begin{equation*}
{\det}_q(a_{ij})=[1,2,\cdots,2n,2n+1,2n+2],
\end{equation*}
where
\begin{align*}
&[i,j]=b_{ij}=\sum_{m=1}^n(a_{2m-1,i}a_{2m,j}-qa_{2m-1,j}a_{2m,i}),\\
&[i,2n+2]=a_{2n+1,i},
\end{align*}
where $1\leq i<j\leq 2n+1$.
\end{corollary}
\begin{proof}
Let $a_{2n+2,i}=a_{i,2n+2}=0,1\leq i\leq 2n+1$ ,then $a_{ij}'s$ satisfied the relations of the quantum coordinate ring $A$. The corollary can be deduced from the previous theorem.
\end{proof}

If we let
\begin{equation}\label{e:col-b}
b'_{ij}=\sum_{m=1}^n(a_{i,2m-1}a_{j,2m}-qa_{j,2m-1}a_{i,2m}),1\leq i<j\leq 2n,
\end{equation}
we also can get the following theorem.

\begin{theorem} In the quantum coordinate ring $M_q(2n)$ we have
$$
\Pf_q(B')={\det}_q(A).
$$
\end{theorem}
\begin{proof} As this is similar to the row case, we only outline the steps.
First we check that the elements $b'_{ij}$ satisfy the relation (\ref{reb}). Then we define
the two-form
\begin{equation}
\Omega=\omega_1'\wedge\omega_2'+\omega_3'\wedge\omega_4'+\cdots+\omega_{2n-1}'\wedge\omega_{2n}',
\end{equation}
where $\omega_i'=\Sigma_{j=1}^{2n} a_{ji}x_j,(i =1,2, . . . , 2n)$.
It is clear thar for this $\Omega$,
$$\bigwedge^n\Omega=(\sum_{\sigma\in S_n}q^{4l(\sigma)})\mbox{det}_q(a_{ij})x_1 \wedge x_2 \wedge \cdots\wedge x_{2n}.$$
On the other hand, it is clear that the $2$-form $\Omega$ can also be expressed as
$$\Omega=\sum_{1\leq i<j\leq 2n}b'_{i,j}x_i\wedge x_j,$$
thus we have
$$\bigwedge^n\Omega=(\sum_{\sigma\in S_n}q^{4l(\sigma)})[1,2,\cdots,2n]x_1 \wedge x_2\wedge\cdots\wedge x_{2n},$$
where $[i,j]=b'_{i,j}$. Comparing these two equations we deduce that
$$\mbox{det}_q(A)=[1,2,\cdots,2n].$$
\end{proof}

\begin{corollary} An odd-order $q$-determinant also has the $q$-Pfaffian representation.
$${\det}_q(a_{ij})=[1,2,\cdots,2n,2n+1,2n+2],$$
where
\begin{align*}
&[i,j]=b'_{i,j}=\sum_{m=1}^n(a_{i,2m-1}a_{j,2m}-qa_{j,2m-1}a_{i,2m}),\\
&[i,2n+2]=a_{2n+1,i},
\end{align*}
here $1\leq i<j\leq 2n+1$. i.e., a $(2n+1)$th-order $q$-determinant
can be expressed as a $(n+1)$th-order $q$-Pfaffian.
\end{corollary}

\section{Quantum hyper-Pfaffians}
From now on, we assume that $q^{8(^k_2)}=1$, and $m=2k$ for $k\in \mathbb Z^+$.
Let $N=\{1,2,\cdots,mn\}$, $S$ be a subset of $N$ and $|S|=mp$. We define two
collections of partitions of $S$ into unordered (resp. ordered) $m$-element subsets as follows.
\begin{align*}
&\Pi(S)=\{(i_1^1,i_1^2,\cdots,i_1^m), \cdots, (i_{p}^1,i_{p}^2,\cdots,i_{p}^m);i_j^l<i_j^{l+1},i_j^1<i_{j+1}^1,i_j^l\in S\},\\
&\Pi(S)'=\{(i_1^1,i_1^2,\cdots,i_1^m), \cdots, (i_{p}^1,i_{p}^2,\cdots,i_{p}^m);i_j^l<i_j^{l+1},i_j^l\in S\}.\\
\end{align*}

We denote $\Pi=\Pi(N)$, $\Pi'=\Pi(N)'$.
Generally, we define $B$ to be the algebra generated by the
$b_{i_1,i_2,\cdots,i_m}$ for $i_1<i_2<\cdots<i_m$ modulo the ideal generated by the relations
\begin{equation}\label{relationb}
\sum_{\pi\in\Pi(S)}(-q)^{l(\pi)}b_{i_1,\cdots,i_m}b_{j_1,\cdots,j_m}=\sum_{\pi\in\Pi(S)}(-q)^{-l(\pi)}
b_{j_1,\cdots,j_m}b_{i_1,\cdots,i_m}
\end{equation}
for all subset $S$ of $\{1,2,\cdots,mn\}$ such that $|S|=2m$.
We define the quantum hyper-Pfaffian\ $[1,2,\cdots ,mn]$ inductively as follows. Let $[i_1,i_2,\cdots,i_m]_m=b_{i_1,i_2,\cdots,i_m}$, then the quantum hyper-Pfaffian
\begin{align}\nonumber
&[1,2,\cdots ,mn]_m\\
&=\sum_{i_2<\cdots<i_m}(-q)^{\sum_{t=2}^{m}(i_t-t)}[1,i_2,\cdots,i_m]_m
[2,\cdots,\hat{i_2},\cdots\hat{i_m},\cdots,mn]_m. \label{e:hyperPh}
\end{align}

\begin{proposition} One has that
\begin{align*}
[1,2,\cdots ,mn]_m&=\sum_{\pi \in\Pi}(-q)^{l(\pi)}b_{\pi}\\
&=\sum_{\pi \in\Pi}(-q)^{l(\pi)}[i_1^1,i_1^2,\cdots,i_1^m]_m\cdots[i_n^1,i_n^2,\cdots,i_n^m]_m.
\end{align*}
\end{proposition}

\begin{proof}
We prove this by induction on $n$. The case of $n=1$ is obvious.
Suppose the lemma is true for $n-1$, then
\begin{align*}
&\sum_{i_2<\cdots<i_m}(-q)^{\sum_{t=2}^{m}(i_t-t)}[1,i_2,\cdots,i_m]_m[2,3,\cdots,\hat{i_2},\cdots\hat{i_m},\cdots,mn]_m\\
=&\sum_{i_2<\cdots<i_m}(-q)^{\sum_{t=2}^{m}(i_t-t)}b_{1,i_2,\cdots,i_m}
\sum_{\pi\in\Pi(N\backslash\{1,i_2,\cdots,i_m \})}(-q)^{l(\pi)}b_{\pi}\\
=&\sum_{i_2<\cdots<i_m}\sum_{\pi_1\in\Pi(N\backslash\{1,i_2,\cdots,i_m \})}(-q)^{{\sum_{t=2}^{m}(i_t-t)}+l(\pi_1)}[1,i_2,\cdots,i_m]_m\cdots[i_n^1,\cdots,i_n^m]_m\\
=&\sum_{\pi \in\Pi}(-q)^{l(\pi)}[i_1^1,i_1^2,\cdots,i_1^m]_m[i_2^1,i_2^2,\cdots,i_2^m]_m\cdots[i_n^1,i_n^2,\cdots,i_n^m]_m\\
=&\sum_{\pi \in\Pi}(-q)^{l(\pi)}b_{\pi}.
\end{align*}
\end{proof}

\begin{lemma}
\begin{align*}
&\sum_{i_1<\cdots<i_m}(-q)^{\sum_{t=1}^{m}(i_t-t)}[i_1,\cdots,i_m]_m[1,\cdots,\hat{i_1},\cdots,\hat{i_m},\cdots ,mn]_m\\
=&(\sum_{i=0}^{n-1}(-q)^{m^2i})[1,2,\cdots,mn]_m.
\end{align*}
\end{lemma}

\begin{proof} Similarly this is proved by induction on $n$.
The case of $n=2$ is true obviously. Suppose it is true for $n-1$,
\begin{align*}
&\sum_{i_1<i_2<\cdots<i_m}(-q)^{\sum_{t=1}^{m}(i_t-t)}[i_1,i_2,\cdots,i_m]_m[1,\cdots,\hat{i_1},\cdots,\hat{i_m},\cdots ,mn]_m\\
=&\sum_{1=i_1<i_2<\cdots<i_m}(-q)^{\sum_{t=1}^{m}(i_t-t)}[1,i_2,\cdots,i_m]_m[1,\cdots,\hat{i_1},\cdots,\hat{i_m},\cdots ,mn]_m\\
+&\sum_{1<i_1<i_2<\cdots<i_m}(-q)^{\sum_{t=1}^{m}(i_t-t)}[i_1,i_2,\cdots,i_m]_m[1,\cdots,\hat{i_1},\cdots,\hat{i_m},\cdots ,mn]_m.
\end{align*}
By definition we have that
\begin{align*}
&\sum_{1=i_1<\cdots<i_m}(-q)^{\sum_{t=1}^{m}(i_t-t)}[i_1,\cdots,i_m]_m[1,\cdots,\hat{i_1},\cdots,\hat{i_m},\cdots ,mn]_m\\
=&[1,2,\cdots ,mn]_m.
\end{align*}
Let $S$ be a subset of $\{1,2,\cdots,n\}$ such that $|S|=2m$, for any
$$\pi=\{(i_1,i_2,\cdots,i_m), (j_1,j_2,\cdots,j_m)\}\in\Pi(S)',$$
we define the length function $l(\pi)$ to be that of the associated permutation, then
it is easy to see that
$$l((i_1,i_2,\cdots,i_m),(j_1,j_2,\cdots,j_m))=m^2-l((j_1,j_2,\cdots,j_m),(i_1,i_2,\cdots,i_m)).$$
Subsequently one has that
\begin{align*}
&\sum_{1<i_1<\cdots<i_m}(-q)^{\sum_{t=2}^{m}(i_t-t)}[i_1,i_2,\cdots,i_m]_m[1,\cdots,\hat{i_1},\cdots,\hat{i_m},\cdots ,mn]_m\\
=&\sum_{S}\sum_{\pi_1\in\Pi(S)}\sum_{\pi_2\in N\backslash S}(-q)^{m^2-l(\pi_1)+l(\pi_2)+\sum_{i\in S}i-m(2m+1)}
b_{j_1,j_2,\cdots,j_m}b_{i_1,i_2,\cdots,i_m}b_{\pi_2}\\
=&(-q)^{m^2}\sum_{S}\sum_{\pi_1\in\Pi(S)}\sum_{\pi_2\in N\backslash S}(-q)^{l(\pi_1)+l(\pi_2)+\sum_{i\in S}i-m(2m+1)}
b_{\pi_1}b_{\pi_2}\\
=&(-q)^{m^2}\sum_{1=i_1<\cdots<i_m}(-q)^{\sum_{t=1}^m(i_t-t)}b_{i_1,\cdots,i_m}.
\end{align*}
Similarly one also has that
\begin{align*}
&\sum_{\substack{j_1<\cdots<j_m\\i_k\neq j_l}}
\sum_{\pi\in\Pi(N\backslash \{i_1,\cdots,i_m,j_1,\cdots,j_m\})}(-q)^{l((j_1,\cdots,j_m)\pi)}
b_{j_1,j_2,\cdots,j_m}b_{\pi}
&\\
=&(-q)^{m^2}\sum_{1=i_1<\cdots<i_m}(-q)^{\sum_{t=1}^m(i_t-t)}b_{i_1,\cdots,i_m}
\sum_{i=0}^{n-2}(-q)^{m^2i}[2,\cdots,\hat{i_2},\cdots,\hat{i_m},\cdots ,mn]_m\\
=&\sum_{i=1}^{n-1}(-q)^{m^2i}[1,2,\cdots ,mn]_m.
\end{align*}

Thus we conclude that
\begin{align*}
&\sum_{i_1<i_2<\cdots<i_m}(-q)^{\sum_{t=2}^{m}(i_t-t)}[i_1,i_2,\cdots,i_m]_m[1,\cdots,\hat{i_1},\cdots,\hat{i_m},\cdots ,mn]_m\\
=&(\sum_{i=0}^{n-1}(-q)^{m^2i})[1,2,\cdots ,mn]_m.
\end{align*}
\end{proof}

\begin{theorem} We have that
\begin{align*}
&\sum_{\pi \in\Pi'}(-q)^{l(\pi)}[i_1^1,i_1^2,\cdots,i_1^m]_m[i_2^1,i_2^2,\cdots,i_2^m]_m\cdots[i_n^1,i_n^2,\cdots,i_n^m]_m\\
=&({\sum_{\sigma \in S_n} (-q)^{m^2l(\sigma)}})[1,2,\cdots,mn]_m.
\end{align*}
\end{theorem}

\begin{proof} When $n=2$ it is obvious. Suppose it is true for $n-1$, then
\begin{align*}
&\sum_{\pi \in\Pi'}(-q)^{l(\pi)}[i_1^1,i_1^2,\cdots,i_1^m]_m[i_2^1,i_2^2,\cdots,i_2^m]_m\cdots[i_n^1,i_n^2,\cdots,i_n^m]_m\\
=&\sum_{i_1<\cdots<i_m}\sum_{\pi\in\Pi(N\backslash\{i_1,i_2,\cdots,i_m\})'}
(-q)^{l(\pi)+\sum_{t=1}^{m}(i_t-t)}b_{i_1,\cdots, i_m}b_{\pi}\\
=&\sum_{i_1<\cdots<i_m}(-q)^{\sum_{t=1}^{m}(i_t-t)}[i_1,\cdots,i_m]\sum_{\sigma\in S_{n-1}}(-q)^{m^2l(\sigma)}[1,\cdots,\hat{i_1},\cdots,\hat{i_m},\cdots ,mn]_m\\
=&(\sum_{\sigma\in S_{n-1}}(-q)^{m^2l(\sigma)})(\sum_{i=0}^{n-1}(-q)^{m^2i})[1,2,\cdots ,mn]_m\\
=&(\sum_{\sigma\in S_{n}}(-q)^{m^2l(\sigma)})[1,2,\cdots ,mn]_m.
\end{align*}
\end{proof}

The following theorem shows that the quantum hyper-Pfaffian is a volume element.
\begin{theorem}
Let $$\Omega=\sum_{1\leq i_1<i_2<\cdots<i_m\leq 2n}b_{i_1,i_2,\cdots,i_m}(x_{i_1}\wedge x_{i_2}\wedge\cdots\wedge x_{i_m}),$$
then one has
$$\wedge ^n\Omega=({\sum_{\sigma \in S_n} (-q)^{m^2l(\sigma)}})[1,2,\cdots ,mn]_mx_1\wedge x_2\wedge\cdots \wedge x_{mn}.$$
\end{theorem}

\begin{proof}
\begin{align*}
\bigwedge ^n\Omega=&\sum_{\pi \in\Pi'}(-q)^{l(\pi)}[i_1^1,i_1^2,\cdots,i_1^m]_m\cdots[i_n^1,i_n^2,\cdots,i_n^m]_mx_1\wedge x_2\wedge\cdots \wedge x_{mn}\\
=&({\sum_{\sigma \in S_n} (-q)^{m^2l(\sigma)}})[1,2,\cdots,mn]_mx_1\wedge x_2\wedge\cdots \wedge x_{mn}.
\end{align*}
\end{proof}

Similar to the case of Pfaffians, a $mn$-th determinant can be expressed as $[1,2,\cdots,mn]_m$, this is proved by the following theorem.
\begin{theorem}\label{qhp}
If  $[i_1,i_2,\cdots,i_m]_m=b_{i_1,i_2,\cdots,i_m}=\sum_{i=1}^n\xi^{m(i-1)+1,\cdots,mi}_{i_1,\cdots,i_m}$,then
$$[1,2,\cdots ,mn]_m={\det}_q(a_{ij})_{1\leq i,j\leq mn}.$$
\end{theorem}

\begin{proof}
We consider the special $m$-form
$$\Omega=\omega_1\wedge\omega_2\wedge\cdots\wedge\omega_m+\omega_{m+1}\wedge\omega_{m+2}\wedge\cdots\wedge\omega_{2m}+\cdots+\omega_{mn-(m-1)}\wedge\cdots\wedge\omega_{mn},$$
where $\omega_i=\Sigma_{j=1}^{mn} a_{ij}x_j,(i =1,2, . . . , mn)$.

For this $\Omega$,
\begin{equation}\label{qhpo1}
\begin{split}
\bigwedge^n\Omega=&(\sum_{\sigma\in S_n}(-q)^{m^2l(\sigma)})\omega_1\wedge\omega_2\wedge\cdots\wedge\omega_{mn}\\
=&(\sum_{\sigma\in S_n}(-q)^{m^2l(\sigma)})\mbox{det}_q(a_{ij})x_1\wedge x_2\wedge\cdots \wedge x_{3n}.\\
\end{split}
\end{equation}

On the other hand,
$$\Omega=\sum_{i_1<i_2<\cdots<i_m}b_{i_1,i_2,\cdots,i_m}x_{i_1}\wedge x_{i_2}\wedge\cdots\wedge x_{i_m},$$
where $b_{i_1,i_2,\cdots,i_m}=\sum_{i=1}^n\xi^{m(i-1)+1,\cdots,mi}_{i_1,\cdots,i_m}.$

In order to prove the theorem we have to check that $b_{i_1,i_2,\cdots,i_m}$ satisfy the relation (\ref{relationb}).
Since Propositions (\ref{thp1}) and (\ref{thp3}) are independent from rows and columns, we only need to show that
Eq. (\ref{relationb})
holds for $S=\{1,2,\cdots,2m\}$.

For two ordered $m$-element subsets
\begin{align*}
\pi&=\{(i_1,\cdots,i_m), (j_1,\cdots,j_m)\}\in\Pi(S), \\
\pi_1&=\{(j_1,\cdots,j_m), (i_1,\cdots,i_m)\},
\end{align*}
we already knew that $l(\pi_1)=m^2-l(\pi).$
Note that $m=2k$,$q^{8(^k_2)}=1$. $m^2-2m=4k(k-1)=8(^k_2)$, then $(-q)^{(m^2-2m)}=1$.
From Proposition \ref{thp1}--\ref{thp3}, we have the following identities:
\begin{equation*}
\begin{split}
\sum_{\pi\in\Pi(S)}(-q)^{l(\sigma)}\xi^{m+1,\cdots,2m}_{i_1,\cdots,i_m}\xi^{1,\cdots,m}_{j_{1},\cdots,j_{m}}
=\sum_{\pi\in\Pi(S)}(-q)^{m^2-l(\sigma)}\xi^{1,\cdots,m}_{j_{1},\cdots,j_{m}}\xi^{m+1,\cdots,2m}_{i_1,\cdots,i_m}.
\end{split}
\end{equation*}
\begin{align}
&\sum_{1=i_1<i_2<\cdots<i_m}(-q)^{l(\sigma)}
\xi^{1,\cdots,m}_{i_1,\cdots,i_m}\xi^{1,\cdots,m}_{i_{n+1},\cdots,i_{2m}}=0, \\
&\sum_{1=i_1<i_2<\cdots<i_m}(-q)^{-l(\sigma)}
\xi^{1,\cdots,m}_{i_{n+1},\cdots,i_{2m}}\xi^{1,\cdots,m}_{i_1,\cdots,i_m}=0.
\end{align}
Then we have
\begin{equation*}
\begin{split}
&\sum_{\pi\in\Pi(S)}(-q)^{l(\pi)}b_{i_1,i_2,\cdots,i_m}b_{j_1,j_2,\cdots,j_m}\\
=&\sum_{\pi\in\Pi(S)}(-q)^{l(\pi)}\sum_{i=1}^n\xi^{m(i-1)+1,\cdots,mi}_{i_1,\cdots,i_m}\sum_{i=1}^n\xi^{m(i-1)+1,\cdots,mi}_{j_1,\cdots,j_m}\\
=&\sum_{\pi\in\Pi(S)}(-q)^{l(\pi)}\sum_{s<t}
(\xi^{m(s-1)+1,\cdots,ms}_{i_1,\cdots,i_m}\xi^{m(t-1)+1,\cdots,mt}_{j_1,\cdots,j_m}+\xi^{m(t-1)+1,\cdots,mt}_{i_1,\cdots,i_m}\xi^{m(s-1)+1,\cdots,ms}_{j_1,\cdots,j_m})\\
=&\sum_{s<t}\sum_{\pi\in\Pi(S)}(-q)^{l(\pi)}
(\xi^{m(s-1)+1,\cdots,ms}_{i_1,\cdots,i_m}\xi^{m(t-1)+1,\cdots,mt}_{j_1,\cdots,j_m}+\xi^{m(t-1)+1,\cdots,mt}_{i_1,\cdots,i_m}\xi^{m(s-1)+1,\cdots,ms}_{j_1,\cdots,j_m})\\
=&\sum_{s<t}\sum_{\pi\in\Pi(S)}(-q)^{l(\pi)}
\xi^{m(s-1)+1,\cdots,ms}_{i_1,\cdots,i_m}\xi^{m(t-1)+1,\cdots,mt}_{j_1,\cdots,j_m}\\
&+\sum_{s<t}\sum_{\pi\in\Pi(S)}(-q)^{m^2-l(\pi)}
\xi^{m(s-1)+1,\cdots,ms}_{j_1,\cdots,j_m}\xi^{m(t-1)+1,\cdots,mt}_{i_1,\cdots,i_m}\\
=&\sum_{s<t}\xi^{m(s-1)+1,\cdots,ms,m(t-1),\cdots,mt}_{1,2,\cdots\cdots,2m}.
\end{split}
\end{equation*}
On the other hand we also have
\begin{align*}
&\sum_{\pi\in\Pi(S)}(-q)^{-l(\pi)}b_{j_1,j_2,\cdots,j_m}b_{i_1,i_2,\cdots,i_m}\\
=&\sum_{\pi\in\Pi(S)}(-q)^{-l(\pi)}\sum_{i=1}^n\xi^{m(i-1)+1,\cdots,mi}_{j_1,\cdots,j_m}\sum_{i=1}^n\xi^{m(i-1)+1,\cdots,mi}_{i_1,\cdots,i_m}\\
=&\sum_{\pi\in\Pi(S)}(-q)^{-l(\pi)}\sum_{s<t}
(\xi^{m(t-1)+1,\cdots,mt}_{j_1,\cdots,j_m}\xi^{m(s-1)+1,\cdots,ms}_{i_1,\cdots,i_m}+\xi^{m(s-1)+1,\cdots,ms}_{j_1,\cdots,j_m}\xi^{m(t-1)+1,\cdots,mt}_{i_1,\cdots,i_m})\\
=&\sum_{s<t}\sum_{\pi\in\Pi(S)}(-q)^{-l(\pi)}
(\xi^{m(t-1)+1,\cdots,mt}_{j_1,\cdots,j_m}\xi^{m(s-1)+1,\cdots,ms}_{i_1,\cdots,i_m}+\xi^{m(s-1)+1,\cdots,ms}_{j_1,\cdots,j_m}\xi^{m(t-1)+1,\cdots,mt}_{i_1,\cdots,i_m})\\
=&(-q)^{-m^2}\sum_{s<t}\sum_{\pi\in\Pi(S)}(-q)^{m^2-l(\pi)}
\xi^{m(t-1)+1,\cdots,mt}_{j_1,\cdots,j_m}\xi^{m(s-1)+1,\cdots,ms}_{i_1,\cdots,i_m}\\
&+(-q)^{-m^2}\sum_{s<t}\sum_{\pi\in\Pi(S)}(-q)^{m^2-l(\pi)}
\xi^{m(s-1)+1,\cdots,ms}_{j_1,\cdots,j_m}\xi^{m(t-1)+1,\cdots,mt}_{i_1,\cdots,i_m}\\
=&(-q)^{-m^2}\sum_{s<t}\sum_{\pi\in\Pi(S)}(-q)^{m^2-l(\pi)}
\xi^{m(t-1)+1,\cdots,mt}_{j_1,\cdots,j_m}\xi^{m(s-1)+1,\cdots,ms}_{i_1,\cdots,i_m}\\
&+(-q)^{-m^2}\sum_{s<t}\sum_{\pi\in\Pi(S)}(-q)^{l(\pi)}
\xi^{m(t-1)+1,\cdots,mt}_{i_1,\cdots,i_m}\xi^{m(s-1)+1,\cdots,ms}_{j_1,\cdots,j_m}\\
=&\sum_{s<t}\xi^{m(s-1)+1,\cdots,ms,m(t-1),\cdots,mt}_{1,2,\cdots\cdots,2m}.
\end{align*}

So we obtain that
$$\sum_{\pi\in\Pi(S)}(-q)^{l(\pi)}b_{i_1,i_2,\cdots,i_m}b_{j_1,j_2,\cdots,j_m}
=\sum_{\pi\in\Pi(S)}(-q)^{-l(\pi)}b_{j_1,j_2,\cdots,j_m}b_{i_1,i_2,\cdots,i_m}.$$
Therefore we have that
\begin{equation}\label{qhpo2}
\bigwedge ^n\Omega=({\sum_{\sigma \in S_n} (-q)^{m^2l(\sigma)}})[1,2,\cdots ,mn]_mx_1\wedge x_2\wedge\cdots \wedge x_{mn}.
\end{equation}
Comparing Eq. (\ref{qhpo1}) and Eq. (\ref{qhpo2}), we deduce that
$$[1,2,\cdots ,mn]_m=\mbox{det}_q(a_{ij})_{1\leq i,j\leq mn}.$$
\end{proof}

The theorem also implies that any $mn$th-order $q$-determinant can be expressed as an $n$-th order $q$-hyper-Pfaffian.

\begin{corollary} Any $(mn+l)$-order $q$-determinant can also be expressed as a $q$-hyper-Pfaffian:
$${\det}_q(a_{ij})=[1,2,\cdots,m(n+1)-1,m(n+1)]_m$$
where $1\leq l\leq m-1$. $a_{mn+i,j}=a_{j,mn+i}=0$, if $i>l$ and $mn+i\neq j$, $a_{mn+i,mn+i}=1$ if $i>l$

$$[i_1,i_2,\cdots,i_m]=b_{i_1,i_2,\cdots,i_m}=\sum_{i=1}^{n+1}\xi^{m(i-1)+1,\cdots,mi}_{i_1,\cdots,i_m}$$
that is,a $(mn+l)$-order $q$-determinant can be expressed as a $(n+1)$th-order generalized $q$-Pfaffian.
\end{corollary}

\begin{theorem}
If we let
$$[i_1,i_2,\cdots,i_m]_m=b_{i_1,i_2,\cdots,i_m}=\sum_{i=1}^{n}\xi_{m(i-1)+1,\cdots,mi}^{i_1,\cdots,i_m},
$$
we also have the following identity
$$[1,2,\cdots ,mn]_m={\det}_q(a_{ij})_{1\leq i,j\leq mn}$$
\end{theorem}

\begin{proof}
This proof is similar to that of Theorem \ref{qhp}. One just uses Remark (\ref{remark1}) to show that
$b_{i_1,i_2,\cdots,i_m}=\sum_{i=1}^{n}\xi_{m(i-1)+1,\cdots,mi}^{i_1,\cdots,i_m}$
satisfy relation (\ref{relationb}). Thus the theorem is proved.
\end{proof}

\begin{corollary} An $(mn+l)$-order $q$-determinant also has the $q$-Pfaffian representation.
$${\det}_q(a_{ij})=[1,2,\cdots,m(n+1)-1,m(n+1)]_m$$
where $1\leq l\leq m-1$. Here $a_{mn+i,j}=a_{j,mn+i}=0$, if $i>l$ and $mn+i\neq j$, $a_{mn+i,mn+i}=1$ if $i>l$
and
$$[i_1,i_2,\cdots,i_m]=b_{i_1,i_2,\cdots,i_m}=\sum_{i=1}^{n+1}\xi_{m(i-1)+1,\cdots,mi}^{i_1,\cdots,i_m}.$$
\end{corollary}

\centerline{\bf Acknowledgments}
The first named author
gratefully acknowledges the support of Max-Planck Institut f\"ur Mathematik in Bonn, Simons Foundation,
NSFC
and NSF.

 \vskip 0.1in

\bibliographystyle{amsalpha}

\end{document}